\documentclass{amsart} 

\usepackage{main}

\title{Singular rational curves on elliptic K3 surfaces}
\author{Jonas Baltes}

\begin{document} 
	\begin{abstract}
		We show that on every elliptic K3 surface there are rational curves $(R_i)_{i\in \NN}$  such that $R_i^2 \to \infty$, i.e., of unbounded arithmetic genus. Moreover, we show that the union of the lifts of these curves to $\PP(\Omega_X)$ is dense in the Zariski topology. As an application we give a simple proof of a theorem of Kobayashi in the elliptic case, i.e., there are no globally defined symmetric differential forms.
	\end{abstract}
	\maketitle	
	\section{Introduction}
		Let $X$ be a complex projective K3 surface. A recent result by Chen--Gounelas--Liedtke \cite{chen2019curves} completed the proof of the conjecture that there are infinitely many rational curves on $X$. Their method also provides information on the classes of these curves in the Picard group if the Picard rank is small. Unfortunately, the folklore conjecture that for every ample class $H\in \Pic X$ there are infinitely many rational curves in $\bigcup_m |mH|$ still remains unknown even for small ranks greater or equal to $2$. In {\itshape loc.\ cit.}\  the following weaker question is posed.
		\begin{question*}
			Does every projective K3 surface $X$ admit rational curves $R_i\subset X$ such that $\lim_i R_i^2 = \infty$?
		\end{question*}
		As there are infinitely many rational curves the question has a positive answer as long as $|\mathrm{Aut}(X)|<\infty$: For a fixed even natural number $2d \in 2\NN$ there are only finitely many orbits of classes $[C]\in \Pic X$ with $C$ an irreducible curve and $C^2= 2d$ under the action of the automorphism group. Moreover the techniques of {\itshape loc.\ cit.}\ prove the question for Picard ranks $1$ and $2$ as well. In this paper we will answer the question positively in the case of elliptic K3 surfaces, too.
		\begin{theorem}
			\label{thm:MainTheorem1}
			Let $X\to\PP^1$ be an elliptic K3 surface. Then there are rational curves $R_i \subset X$ such that $R_i^2 \to \infty$. 
		\end{theorem}
		In other words, the only missing cases are non-elliptic K3 surfaces of Picard rank $3$ or $4$ with infinite automorphism group.
		
		The method of the proof of Theorem \ref{thm:MainTheorem1} builds on the techniques by Bogomolov--Tschinkel \cite{bogomolov1999density} and Hassett \cite{hassett2003potential} who constructed infinitely many rational curves on a complex elliptic K3 surface. Their results have since also been extended to  characteristic  $p>3$ by Tayou in \cite{tayou2018rational}. The main idea is to start with a rational curve $R$ and look at its image under certain rational maps between elliptic K3 surfaces.
		As it turns out the main problem faced in these papers is that the initial rational curve $R$ might be torsion, which prevents the images from giving \emph{new} curves. Here, torsion means  that for any two points in $R\cap X_t$ of a smooth fiber their difference in $\mathrm{Jac}^0(X_t)$ is torsion.
		
		In our case we look at the same construction and examine when the image of the curve $R$ will have more singularities. What prevents the images from doing so is very similar to being a torsion section which leads to our main definition of \emph{quasi-torsion sections}, see Section \ref{sec:QuasiTorsion}. The existence of rational non-quasi-torsion curves will then be carried out in Section \ref{sec:ExistenceRNQT} which will then be needed to produce the rational curves with an unbounded number of singularities in Section \ref{sec:ProducingSingularities}.
		
		In Section \ref{sec:JetSpace} we will apply the methods to examine lifts of rational curves in the first jet space $\PP(\Omega_X)$ by which we mean the space of one-dimensional quotients of $\Omega_X$. Recall the construction of such lifts: For every curve $C\subset X$ and its normalization $f\colon \tilde{C}\to C\hookrightarrow X$ the usual short exact sequence of cotangent bundles gives a map
		\begin{equation*}
			f^*\Omega_X\to \Omega_{\tilde{C}}^1.
		\end{equation*}
		Denote its torsion free image by $L$ which is automatically a line bundle. Then the surjective map $f^*\Omega_X\twoheadrightarrow L$ gives rise to a lift $\tilde{C}\to \PP(\Omega_X)$. If $C$ is rational then $\deg L <0$ and the lift is negative with respect to $\oO_{\PP(\Omega_X)}(1)$. It turns out that these pathological curves form a dense subset.
		\begin{theorem}
			\label{thm:MainTheorem2}
			Let $X\to\PP^1$ be an elliptic K3 surface. Then the union of lifts of rational curves to the jetspace $\PP(\Omega_X)$ is Zariski-dense.
		\end{theorem}
		In Section \ref{sec:Applications} we give some easy consequences of these results. For example the above mentioned density yields a short proof of Kobayashi's theorem in the elliptic case, see Theorem \ref{thm:Kobayashi}.
		
		In \cite{chen2013density} Chen-Lewis were concerned with the conjecture that the union of rational curves  on $X$ is dense in the \emph{usual} topology. For elliptic K3 surfaces they proved this as long as there exists a rational multisection on $X$ that is not torsion. As a by-product of our theorems we see that the elliptic structure can be changed in such a way that there exists such a multisection and hence density of rational curves holds for \emph{every} elliptic K3 surface, see Corollary \ref{cor:DensityUsual}.\\
		\textbf{Notations: }Let $p\colon X\to B$ be an elliptic fibration and $U\subset B$ be the subset on which the fibration is smooth. By $(\;)_U$ we mean the restriction to $p^{-1}(U)$. If the fibration is moreover Jacobian, i.e., it admits a section, then we denote the closure of the $m$-torsion of the fibers by $X[m]$. The upper halfplane in $\CC$ is denoted by $\HH$.\\
		\textbf{Acknowledgements: }The author would like to thank his adviser Frank Gounelas for many hours of helpful discussions and his advice on several draft versions. 
	\section{Background on Elliptic K3 surfaces and Jacobians}
	\label{sec:background}
		%Fibers, Cmpactified Jacobian
		We start by collecting facts on elliptic K3 surfaces, which we always assume to be projective. For a detailed discussion, see \cite[Chapter 11]{huybrechts2016lectures}.
		
		Let $X\to\PP^1$ be an elliptic K3 surface. Its index $d_0\in \NN$ is defined as
		\begin{equation*}
		    d_0 = \min\{0< c_1(L).X_t \eft L\in\Pic X\} = \min\{0\neq C.X_t \eft C\subset X \;\text{a curve}\},
		\end{equation*} 
		where the last equation follows as $c_1(L)+nX_t$ becomes effective for $n\gg 0$.
		\subsection{Compactified Jacobians} 
		\label{sub:compactifiedJacobian}
		Denote by $\mathrm{Jac}^d (X/\PP^1)\to\PP^1$ the relative Jacobian of the elliptic fibration. Then we can define the compactified Jacobian $J^d(X)\to\PP^1$ as the unique relatively minimal smooth model of $\mathrm{Jac}^d(X/\PP^1)\to\PP^1$. Therefore over the smooth fibers one recovers $J^d(X)_t \cong \mathrm{Jac}^d(X_t)$, where the latter is the usual Jacobian of a curve. By \cite[Prop. 11.4.5]{huybrechts2016lectures} all compactified Jacobians are K3 surfaces as well and moreover for every $n\in \NN$ we can find another elliptic K3 surface $Y\to\PP^1$ such that there is an isomorphism $J^n(Y) \cong X$ as elliptic surfaces. Moreover the index of $Y$ is exactly $nd_0$, where $d_0$ is the index of $X$.
		
		Furthermore Jacobians give rise to rational maps between elliptic K3 surfaces as follows: For a smooth fiber we have a canonical morphism 
		\begin{equation*}
			\mathrm{Jac}^m(X_t) \times \mathrm{Jac}^n(X_t)\to \mathrm{Jac}^{m+n}(X_t),
		\end{equation*}
		which is given by the tensor product of line bundles. This globalizes to give a rational map
		\begin{equation*}
			J^m(X) \times_{\PP^1} J^n(X)\dashrightarrow J^{m+n}(X),
		\end{equation*}
		which is defined over the smooth locus $U\subset\PP^1$. Using the diagonal morphism we can construct a multiplication map $J^1(X)\dashrightarrow J^n(X)$ for every $n\in \NN$ by mapping
		\begin{equation*}
            J^1(X)\to J^1(X)\times_{\PP^1} \ldots \times_{\PP^1} J^1(X) \dashrightarrow J^{n}(X),		
        \end{equation*}
        where the first map is the diagonal map into the $n$-fold fiberproduct.
		To relate these rational maps to the K3 surface $X$ we mention that the canonical isomorphism $X_t\cong \mathrm{Jac}^1(X_t)$ gives an \emph{isomorphism} $X\to J^1(X)$ respecting the fibration. Moreover choosing a line bundle $M\in \Pic X$ of degree $d_0$ we get another isomorphism 
		\begin{equation*}
			J^n(X) \to J^{n+d_0}(X),
		\end{equation*}
	 	which fiberwise is given by the tensor product with $M$, i.e.,
	 	\begin{equation*}
			L \mapsto L\otimes M|_{X_t}
	 	\end{equation*}
 		for a line bundle $L\in \mathrm{Jac}^n(X_t)$.
	
		\subsection{Framed elliptic curves}
		\label{sub:framedElliptic}
		We recall some standard facts on elliptic curves, see e.g., \cite{hain2014lectures}.
		\begin{definition}
		    A framed elliptic curve is a triple $(E,a,b)$ of a complex elliptic curve $E$ and two elements $a,b\in \hh_1(E,\ZZ)$ such that their intersection is $a\cdot b = 1$. Isomorphisms of framed elliptic curves are isomorphisms of elliptic curves that respect the frame.
		    
		    A framed lattice is a triple $(\Lambda, \lambda_1, \lambda_2)$ such that $\Lambda\subset \CC$ is a rank two lattice and $\lambda_1,\lambda_2\in\Lambda$ is a $\ZZ$-basis of $\Lambda$ with $\Im (\lambda_1/\lambda_2) > 0$. Two framed lattices are isomorphic if the lattice and the frame coincide up to a complex multiple.
		\end{definition}
		For example every family of elliptic curves $F\to B$ over a simply connected base $B$ can be simultaneously framed, i.e., there is a tuple $(a,b)$ in $\hh_1(F,\ZZ)$ such that the pushforward of the frames of every fiber coincide with $(a,b)$.
		
		There is a one-to-one bijection 
		\begin{equation*}
		    \HH \leftrightarrow \left\{\begin{gathered}
            \textit{isomorphism classes} \\
            \textit{of framed lattices}
        \end{gathered}\right\} \leftrightarrow\left\{\begin{gathered}
            \textit{isomorphism classes of} \\
            \textit{framed elliptic curves}
        \end{gathered}\right\}
		\end{equation*}
		which sends some $\tau \in \HH$ to $\Lambda_\tau = \ZZ\tau+\ZZ$ and a lattice $\Lambda$ to $\CC/\Lambda$. Moreover the upper half plane $\HH$ is a fine moduli space for framed elliptic curves with universal curve given by
		\begin{equation*}
		    \eE = \CC\times \HH / \{(\ZZ\tau+\ZZ, \tau)\eft \tau \in \HH\}.
		\end{equation*}
        For a chosen frame $(\Lambda,\lambda_1,\lambda_2)$ there is a natural choice of coordinate function
        \begin{align*}
            \RR^2 &\to \CC/\Lambda\\
            (x,y)&\to x\lambda_1+y\lambda_2, 
        \end{align*}
        which induces a homeomorphism $\RR^2/\ZZ^2 \cong \CC/\Lambda$.
        If we change the frame of $\Lambda$ by an element $\gamma\in \SL(2,\ZZ)$, the corresponding coordinates for $p = x\lambda_1+y\lambda_2$ in the new frame are given by $\gamma^T\cdot \left(\begin{smallmatrix}
            x\\y
        \end{smallmatrix}\right)$, where $\gamma^T$ is the transposed matrix.
        
		\subsection{Singular Fibers}
		\label{sub:singularFibers}
		The singular fibers of elliptic fibrations can be completly understood by means of their local monodromy group, for details see \cite[Lecture IV]{miranda1989basic}. 
		The latter is defined as follows. 
		Pick a small disc $\Delta\subset \PP^1$ such that over the punctured disc the map $X_{\Delta^*}\to\Delta^*$ is smooth and fix a fiber $X_t\cong \CC/(\ZZ+\tau\ZZ)$. 
		Then the usual monodromy action of $\ZZ \cong \pi_1(\Delta^*, t)$ on the first integral cohomology of $X_t$ gives rise to a subgroup $\Gamma\subset \SL(2,\ZZ)$ which is called the \emph{local monodromy group}. 
		
		We just recall the facts that are important to our case, for a complete classification see \cite[Diagram 11.1.3]{huybrechts2016lectures}.
		It turns out that the local monodromy is infinite precisely for the fibers of type $I_n, I_n^*\;(n>0)$, which occur on a K3 surface if and only if the fibration is non-isotrivial. In this case the local monodromy can be generated by the following elements
		\begin{equation*}
			I_n:\;\begin{pmatrix}
				1&n\\0&1
			\end{pmatrix}\qquad\qquad I_n^*:\;-\begin{pmatrix}
			1&n\\0&1
		\end{pmatrix}.
		\end{equation*} 
		
	\section{Quasi-torsion sections}
		\label{sec:QuasiTorsion}
		In the following we will introduce the main definition of this paper, which is a generalization of torsion multisections. 
		Recall the definition of the latter from \cite{bogomolov1999density}.
		\begin{definition}
		    Let $X\to\PP^1$ be an elliptic K3 surface. A multisection $M\subset X$ is called torsion if for any two points $x,y\in M\cap X_t$ in every smooth fiber $X_t$ their difference $x-y\in \mathrm{Jac}^0(X_t)$ is torsion.
		\end{definition}
		Throughout this section we will work in the analytic category unless otherwise stated. 
		
		Let $p\colon X\to\Delta$ a smooth elliptic Jacobian fibration between complex manifolds over a simply connected base $\Delta$. Then a choice of frame for the family yields a holomorphic $\tau\colon \Delta\to \HH= \{z\in \CC\eft \Im z > 0\}$ such that 
		\begin{equation*}
			\label{eq:LocalElliptic}
			X = \CC\times \Delta / (\ZZ\tau(t)+\ZZ,t)
		\end{equation*} 
		and the section is given by $\{0\}\times \Delta$. We call such a choice a \emph{standard model}.
		
		The branches of the $m$-torsion $X[m]$ are of the form $\{(a\tau(t)+b, t)\eft t\in \Delta\}$ for some $a,b\in \frac{1}{m}\ZZ\subset\QQ$. We generalize these multisections in the following way.
		\begin{definition}
			Let $X\to B$ be an elliptic Jacobian fibration between two complex manifolds such that the base $B$ is $1$-dimensional. A holomorphic curve $C\subset X$ is called \emph{elementary quasi-torsion}  if $C_U\to U$ is \'{e}tale and the branches over every simply connected $\Delta\subset U$ and some choice of standard model $X_\Delta = \CC\times \Delta / (\ZZ\tau(t)+\ZZ, t)$ are given by
			\begin{equation*}
				\{(a\tau(t)+b,t)\eft t\in\Delta\} \subset X_\Delta
			\end{equation*}
			for some $a,b\in \RR$ which may depend on $\Delta$ and the chosen standard model.
		\end{definition}
		\begin{remark}
		    The above definition is independent of the choice of standard model: If we have two standard models over $\Delta$ given by $\tau,\tau':\Delta\to \HH$, then $\tau' = \gamma \cdot \tau$ with $\gamma\in \SL(2,\ZZ)$. If we denote $\left(\begin{smallmatrix}
            a^\prime\\b^\prime
            \end{smallmatrix}\right) = (\gamma^{T})^{-1}\cdot \left(\begin{smallmatrix}
            a\\b
            \end{smallmatrix}\right)$ then the two curves
            \begin{align*}
                \{(a\tau(t)+b,t)\eft t\in\Delta\} &\subset \CC\times \Delta / (\ZZ\tau(t)+\ZZ,t)\\
                \{(a'\tau'(t)+b',t)\eft t\in\Delta\} &\subset \CC\times \Delta / (\ZZ\tau'(t)+\ZZ,t)
            \end{align*}
            coincide in $X_\Delta$. Moreover by the same reasoning it suffices to check the conditions only on an open cover of $U$.
		\end{remark}
		\begin{example}
			\label{ex:EQTIsotrivial}
			Let $p\colon X\to \PP^1$ be an isotrivial Jacobian elliptic projective surface with general fiber isomorphic to a fixed elliptic curve $E$. Then there exists a projective curve $C$ and a finite rational morphism
			\begin{equation*}
				\label{eq:RationalMapIsotrival}
				C\times E \dashrightarrow X
			\end{equation*}
			that respects the section and the elliptic structure. The closure of the image of $C\times \{pt\}$ under the rational map above defines an elementary quasi-torsion curve. In fact this is an example of an \emph{algebraic} elementary quasi-torsion curve.
		\end{example}
		\begin{lemma}
		    Let $X\to\PP^1$ be a Jacobian elliptic fibration and $x\in X_U$ a point. Then there exists a unique holomorphic connected elementary quasi-torsion curve inside $X_U$ that contains $x$.
		\end{lemma}
	    \begin{proof}
	        Let $\Delta\subset U$ be a simply connected subset such that $x\in X_\Delta$ and let $X_\Delta\cong \CC\times \Delta/(\ZZ\tau(t)+\ZZ,t)$ be a standard model.
	        Then we can choose $(a,b)\in\RR^2$ such that $x = (a\tau(t_0)+b,t_0)$. Such a choice is unique up to $\ZZ^2$ and hence any branch of an elementary quasi-torsion curve that contains $x\in X_U$ is equal to
	        \begin{equation*}
	            \{(a\tau(t)+b)\eft t\in\Delta\}\subset X_\Delta.
	        \end{equation*}
	        Thus, the uniqueness follows from the curve being \'{e}tale and connected.
	        
	        To construct the curve we denote by $U'\to U$ the universal cover and by $X'$ the pullback of $X_U\to U$ to $U'$. If we choose a standard model 
		    \begin{equation*}
		        \CC\times U'/(\ZZ\tau(t)+\ZZ,t) \cong X'
		    \end{equation*}
		    we may choose a point $x'\in X'$ that lies over $x$ via the map $p\colon X' \to X_U$. Then we may choose $(a,b)\in \RR^2$ such that $x'$ lies in
		    \begin{equation*}
		        T'[x'] := \{(a\tau(t)+b, t)\eft t\in\Delta\}.
		    \end{equation*}
		    We then denote 
		    \begin{equation*}
		        T[x] := p(T'[x'])\subset X_U,
		    \end{equation*}
		    which is a connected elementary quasi-torsion curve containing $x\in X_U$. 
	    \end{proof}	
	    \begin{definition}
	        Let $X\to \PP^1$ be a Jacobian elliptic fibration. For any point $x\in X_U$ the unique holomorphic elementary quasi-torsion curve that contains $x$ is denoted $T[x]$.
	    \end{definition}

		As we have seen in Example~\ref{ex:EQTIsotrivial} in the isotrivial case every $T[x]$ is algebraic and hence extends to a curve on $X$. But as the construction above is very analytic in nature this is not guaranteed in any case. We will see that for non-isotrivial fibrations quite the opposite is true: only those $T[x]$ contained in $X[m]$ for some $m\in \NN$ extend to the whole of $X$. 
		\begin{remark}
			Let $x = a\tau+b \in X_{t_0} = \CC/(\ZZ\tau+\ZZ)$ be an element in a smooth fiber of $X\to \PP^1$. As $T[x]$ is \'{e}tale over $U$ there is a well defined action of $\pi_1(U, t_0)$ on $X_{t_0}$. This action factors through the monodromy group $\Gamma\subset \operatorname{SL}(2,\ZZ)$ by acting on the tuple $(a,b)$ by the right action induced by the transposed matrix.
		\end{remark}
		\begin{proposition}
			\label{prop:NoQuasiTorsion}
			Let $X\to \PP^1$ be a non-isotrivial elliptic projective Jacobian surface. Then for some $x\in X_U$ the holomorphic curve $T[x]\subset X_U$ extends to an algebraic curve on $X$ if and only if $T[x] \subset X[m]$ is torsion for some $m\in \NN$.
		\end{proposition}
		\noindent The main idea of the proof is to show that $|T[x]\cap X_t|=\infty$ for non-torsion points $x\in X_U$. This can be seen as an analogue of the fact that the torsion $X[p]$ without the zero-section is irreducible for $p$ a large prime and $X[p].X_t = p^2$, see e.g., \cite[Theorem 8.3]{hassett2003potential}
		
		To deduce the above statement we make use of the monodromy action, which can be characterized by the following lemma.
		\begin{lemma}[{Hassett \cite[Lemma 8.4, Lemma 8.5]{hassett2003potential}}]
			Let $X\to \PP^1$ be a projective non-isotrivial Jacobian elliptic surface. Then the reduction $\Gamma\subset \SL(2,\ZZ) \to \SL(2,\ZZ/p\ZZ)$ of the monodromy group is surjective for primes $p\gg 0$.
		\end{lemma}
		\begin{proof}[Proof of Proposition \ref{prop:NoQuasiTorsion}]
			Suppose $T[x]$ extends on $X$, i.e., it is algebraic. In particular $|T[x]\cap X_t|$ is finite. As $X$ is non-isotrivial there is a degenerate fiber of type $I_N$ or $I_N^*$. By fixing an appropriate smooth fiber $X_t = \CC/(\ZZ\tau+\ZZ)$ we can assume that 
			\begin{equation*}
				\gamma_n =\begin{pmatrix}
					1&2nN\\0&1
				\end{pmatrix} \in \Gamma
			\end{equation*}
			%\big(\begin{smallmatrix}
			%	1&2nN\\0&1
			%\end{smallmatrix}\big)\in \Gamma$ 
			is contained in the monodromy group $\Gamma$ for every $n\in \NN$. Let $x = a\tau+b \in X_t\cap T[x]$. Then applying $\gamma_n$ yields
			\begin{equation*}
			    \gamma_n^T.\begin{pmatrix}
					a\\b
				\end{pmatrix} = \begin{pmatrix}
					a\\2anN+b
				\end{pmatrix}.
			\end{equation*}
			As the intersection of $T[x]$ with $X_t$ is finite $2anN+b = b\in \RR/\ZZ$ for some $n\in\NN_{>0}$. Therefore we have that $a\in \QQ$ is rational. On the other hand choose $p\gg 0$ such that the previous lemma is fulfilled. Then the matrix
			\begin{equation*}
				\begin{pmatrix}
					pw& 1+px\\
					-1+py& pz
				\end{pmatrix}\in \Gamma
			\end{equation*}
			is contained in the monodromy group for some $w,x,y,z\in \ZZ$. This yields
			\begin{equation*}
			    \begin{pmatrix}
			        pwa+(-1+py)b\\
			        (1+px)a+pzb
			    \end{pmatrix}\in T[x]\cap X_t,
			\end{equation*}
			which then implies that $pwa+(-1+py)b\in\QQ$ is rational as above and hence $b\in\QQ$ is rational as well. 
		\end{proof}

		We will now give a local criterion for a holomorphic curve to be elementary quasi-torsion.
		\begin{proposition}
			\label{prop:CapTorsionEmpty}
			Let $X\to\Delta$ be a standard model and let $I\subset \NN$ be an infinite multiplicatively closed subset. Assume that a section $C\subset X$ of $X\to\Delta$ satisfies
			\begin{equation*}
				C\cap \bigcup_{n\in I} X[n] = \emptyset.
			\end{equation*}
			Then $C$ is elementary quasi-torsion.
		\end{proposition}
		\begin{proof}
			Let 
			\begin{equation*}
			    X' = \CC\times \Delta \to (\CC\times\Delta)/(\ZZ\tau(t)+\ZZ,t)= X
			\end{equation*} 
			be the universal cover of the standard model. As $\Delta$ is simply connected the section $C$ lifts to a section $C'$ of $X'\to \Delta$. By assumption 
			\begin{equation*}
				C' \subset (\CC\times \Delta) \setminus \bigcup_{n\in I} (\tfrac{1}{n}\ZZ\tau(t)+\tfrac{1}{n}\ZZ, t)
			\end{equation*}
			for the infinite multiplicatively closed set $I$.
			Denote by $f\colon  \Delta\to \CC$ a function that induces a chart for the curve $C'\subset \CC\times \Delta$, i.e., $C' = \{(f(t), t)\eft t\in\Delta\}$. Then $f(t) = a(t)\tau(t)+b(t)$ for some continuous real valued functions $a,b\colon \Delta\to \RR$.
			
			We will now use the fact that $\bigcup_{n\in I} X[n]$ is dense in $X$ to show that $a(t)$ and $b(t)$ are constant. 
			
			By contradiction assume that this is not the case,  i.e., without loss of generality $b$ is non-constant and therefore there is some $t_0$ such that $b_0 = b(t_0)\in \tfrac{1}{n}\ZZ$ for some $n\in I$. Then the function
			\begin{equation*}
				F\colon  \CC\times \Delta \to \CC,\; F(z,t) = f(t)- z\tau(t)-b_0
			\end{equation*}
			has a zero at $(a(t_0), t_0)$ and a Jacobian of maximal rank. The implicit function theorem gives an open $t_0\in U\subset \Delta$ and a holomorphic function $g\colon U\to \CC$ such that $f(t) - g(t)\tau(t)-b_0 = 0$ for all $t\in U$. If $g$ is constant we are done, so otherwise the image is open. As $a(t_0)\in \RR$ is contained in the image of $g$ there is an $a_0 = g(t')\in \tfrac{1}{m}\ZZ$ also contained in the image for $m\in I$ large enough. Therefore the point 
			\begin{equation*}
			    (f(t'),t')=(a_0\tau(t') + b_0,t') \in X[nm]\cap C
			\end{equation*} 
			is torsion, a contradiction.
		\end{proof}
		%Maybe change projective and stuff
		\begin{corollary}
			\label{cor:TorsionDense}
			Let $p\colon X\to \PP^1$ be a Jacobian elliptic fibration and $C\subset X$ an irreducible holomorphic curve that is not elementary quasi-torsion. Then the set 
			\begin{equation*}
			    C\cap \bigcup_{n\in I} X[n] \subset C
			\end{equation*}
			is dense. 
		\end{corollary}	
		\begin{proof}
		    Let $V\subset C$ be an open set. By shrinking we may assume it to be simply connected and open. If $V\cap \bigcup_{n\in I} X[n] = \emptyset$ then for $\Delta = p(V)$ the set $V$ is an elementary quasi-torsion curve in $X_\Delta\to\Delta$ by the previous proposition. Hence, $C$ agrees with $T[x]$ on the open set $V$ for some $x\in X_U$ and thus they are equal everywhere.
		\end{proof}

		We now come to the main definition, which is a generalization of torsion multisections. Let $X\to \PP^1$ be a (not necessarily Jacobian) elliptic K3 surface. Then there is the rational difference map to the compactified Jacobian $J^0(X)$:
		\begin{equation*}
			\label{eq:DifferenceMap}
			d\colon  X_U\times_{U} X_U \cong J^1(X)_U \times_{U} J^1(X)_U \to J^0(X),
		\end{equation*}
		where the last arrow maps two line bundles $L,L'\in \Pic X_t$ to $L^{-1}\otimes L'$. 
		\begin{definition}
			Let $C\subset X$ be an irreducible holomorphic curve not contained in a fiber. We define $D(C) = d(C_U\times_{U} C_U)$ and say that $C$ is a \emph{quasi-torsion} multisection if 
			\begin{equation*}
				D(C) = \bigcup_{x\in S} T[x],%\overline{T}[x],
			\end{equation*}
			for some finite subset $S\subset X_U$. Otherwise it is \emph{non-quasi-torsion}.	
		\end{definition}
		
		\begin{remark}
		    Every torsion multisection is quasi-torsion as well as all elementary quasi-torsion curves. 
		\end{remark}
		    
    \section{Existence of rational non-quasi-torsion curves}
        \label{sec:ExistenceRNQT}
        In this section we will prove that there are rational non-quasi-torsion curves on elliptic K3 surfaces, as long as we allow a change of the fibration. The proof will be split into two parts as we have to take care of the isotrivial case seperately.
        
        We will introduce some notation which is taken from \cite[Corollary 9.5]{hassett2003potential} applied to the isotrivial case. If $X\to\PP^1$ is an isotrivial K3 surface with $n_0$ (resp. $n_2, n_3$ and $n_4$) fibers of type $I_0^*$ (resp. type $II, II^*$, type $III$,$III^*$ and type $IV$,$IV^*$), we denote 
        \begin{equation*}
            c(X\to\PP^1) = \tfrac{1}{2}n_0 + \tfrac{5}{6}n_2 + \tfrac{3}{4}n_3 + \tfrac{2}{3}n_4 -2.
        \end{equation*}
        The goal of this section is to prove the following theorem.
        \begin{theorem}
            \label{thm:NQTCurvesExist}
            Let $p\colon X\to\PP^1$ be an elliptic K3 surface. If $X\xrightarrow{p}\PP^1$ is non-isotrivial or isotrivial with $c(X\xrightarrow{p}\PP^1)>0$, then there is a non-quasi-torsion rational curve on $X$. If $p\colon X\to\PP^1$ is isotrivial  with $c(X\xrightarrow{p}\PP^1)\le0$ there is another elliptic fibration $p'\colon X\to\PP^1$ such that the previous conditions hold.
        \end{theorem}
        We start with the latter reduction step by using a similar technique as in \cite{oguiso1989jacobian}, where all Jacobian elliptic pencils on some special elliptic Kummer surfaces are constructed.
		\begin{lemma}
			\label{lem:WLOGC>0}
			Let $p\colon X\to\PP^1$ be an isotrivial elliptic K3 surface with $c(X\xrightarrow{p}\PP^1)\le 0$. Then there is another fibration $p'\colon X\to\PP^1$ that is non-isotrivial or isotrivial with $c(X\xrightarrow{p'}\PP^1)>0$.
		\end{lemma}
        
        \begin{proof}
			By \cite[Proposition 9.6]{hassett2003potential} we have $\operatorname{rk} \Pic X\ge 16$. Hence we can replace $p\colon X\to \PP^1$ by a Jacobian fibration $p'\colon X\to\PP^1$ with a section $S$. If it is isotrivial with $c(X\xrightarrow{p'}\PP^1)\le 0$ then by \cite[Proposition 9.6]{hassett2003potential} the only singular fibers that can occur are as in Figure \ref{diag:DynkinKummerType}. We pick two degenerate fibers $F_1,F_2$ and denote the components as indicated in Figure~\ref{diag:DynkinKummerType}, where $\alpha_1$ denotes the component meeting the section $S$.
			\begin{figure}[H]
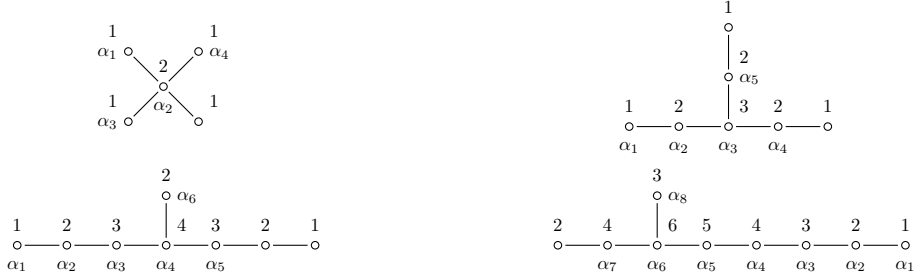

				% minipage mit (Blind-)Text
				\begin{subfigure}[b]{\textwidth}
				\begin{minipage}{0.4\textwidth} 
					\center
					\dynkin[Kac,extended,labels*={1,1,2,1,1}, labels={\alpha_1, \alpha_3,\alpha_2, \alpha_4}] D4
				
				\end{minipage}
				\hfill
				\begin{minipage}{0.4\textwidth}
					\center
					\dynkin[Kac,extended,labels*={1,1,2,3,2,1,2}, labels={,\alpha_1,\alpha_2,\alpha_3, \alpha_4,,\alpha_5}] E6		
				
				\end{minipage}
				\hfill
			    \end{subfigure}
			    \bigskip
				\begin{subfigure}[b]{\textwidth}
				\begin{minipage}{0.4\textwidth}
					\center
					\dynkin[Kac,extended,labels*={1,2,3,4,3,2,1,2}, labels={\alpha_1,\alpha_2,\alpha_3, \alpha_4,\alpha_5,,,\alpha_6}] E7		
				
				\end{minipage}
				\hfill
				\begin{minipage}{0.4\textwidth}
					\center
					\dynkin[Kac,extended,labels*={1,2,3,4,5,6,4,2,3}, labels={\alpha_1,\alpha_2,\alpha_3, \alpha_4,\alpha_5,\alpha_6, \alpha_7, ,\alpha_8}] E8		
				
				\end{minipage}		
				\end{subfigure}
				\caption{Fibers occuring in isotrivial fibrations with $c(X)\le 0$.}
				\label{diag:DynkinKummerType}
			\end{figure}
			\begin{figure}[ht]
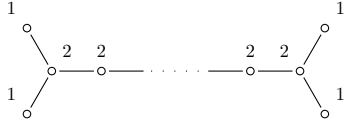

				% minipage mit (Blind-)Text
				\begin{minipage}{0.4\textwidth} 
					\center
					\dynkin[Kac,extended,labels*={1,1,2,2,2,2,1,1}, ] D{}
					
				\end{minipage}
				\caption{Fiber of type $I_n^*$}
			\end{figure}
			For both fibers $F_i\;(i=1,2)$ let $A_i= 2\sum_{j=1}^{k-2} \alpha_j +\alpha_{k-1}+ \alpha_k\in \Pic(X)$, where the components $\alpha_j$ are the components of the respective fiber as indicated above in Figure~\ref{diag:DynkinKummerType} and $k$ is the highest occuring index.
			The effective divisor $E = 2S+A_1+A_2$ defines a nef primitive class with $E.E = 0$. Then $|\oO(E)|$ induces an elliptic fibration by \cite[Proposition 2.3.10]{huybrechts2016lectures} and by  construction the fibration has a fiber of type $I_n^*$ with $n>0$. We conclude that this fibration is non-isotrivial.
		\end{proof}
		
        \subsection*{The non-isotrivial case}
        The non-isotrivial case is particularly simple as there is the following theorem:
        \begin{theorem}[{\cite[Theorem 8.3]{hassett2003potential}}]
            Let $X\to\PP^1$ be a non-isotrivial K3 surface. Then there exist non-torsion rational multisections on $X$.
        \end{theorem}
        
        \begin{proof}[Proof of Theorem \ref{thm:NQTCurvesExist} in the non-isotrivial case.]
            Let $R\subset X$ be a non-torsion rational multisection coming from \cite{hassett2003potential}.
            The difference $D(C)\subset J^0(X)$ yields an algebraic subset with not all of its irreducible components contained in some $X[m]\; (m\in \NN)$. But if $C$ was quasi-torsion, all components would be contained in some $X[m]$ by Proposition \ref{prop:NoQuasiTorsion}, a contradiction.
        \end{proof}
        
        \subsection*{The isotrivial case with \texorpdfstring{$c(X\to\PP^1)>0$}{c(X, P1)>0}}
        We proceed by imitating the genus calculation from \cite{hassett2003potential} in the case of quasi-torsion multisections by investigating the local monodromy. From this we will see that the genus of quasi-torsion curves $C$ grows with its fiber degree $C.X_t$. We will need the following preparatory lemma:
        \begin{lemma}
			Let $\operatorname{id} \neq \gamma\in \SL(2,\ZZ)$ be an element of finite order and $d<\ord(\gamma)$.  Then there is a natural number $\kappa$ such that there exists an $x\in \RR^2\setminus \bigcup_{i=1}^\kappa \tfrac{1}{i}\ZZ^2$ with
			\begin{equation*}
				\sum_{i=0}^d \gamma^i x = 0 \mod \ZZ^2
			\end{equation*}
			if and only if $d = \ord(\gamma)-1$. Moreover in this case $\sum_{i=0}^{\ord(\gamma)-1} \gamma^i = 0$.
		\end{lemma}
		\begin{proof}
			As $\sum_{i=0}^{\ord(\gamma)-1} \gamma^i = 0$ holds, one direction is obvious.
			
			Let $n< \ord(\gamma)$. Then $\id - \gamma^n$ is invertible over $\QQ$ as $\gamma^n$ has no Eigenvalue $1$. Let $A$ be its inverse. Then $B = A\cdot (1-\gamma)$ is an inverse for $C = \sum_{i=0}^n\gamma^n$. Therefore 
			$Cx \in \ZZ^2$ implies 
			\begin{equation*}
					x = BCx \in \frac{1}{|\det \id - \gamma^n|}\ZZ^2.
			\end{equation*}
			Then $\kappa = \max_n |\det \id - \gamma^n|$ yields the result.
		\end{proof}
		The geometric meaning of the lemma is as follows. Recall that an element $\gamma\in\SL(2,\ZZ)$ acts on the points $x=a\tau+b\in \CC/(\ZZ\tau+\ZZ)$ of an elliptic curve $E= \CC/(\ZZ\tau+\ZZ)$ by acting on the tuple $(a,b)$ via the transposed matrix.
		\begin{corollary}
		    \label{cor:OrderOfPointInElliptic}
		    Let $\operatorname{id} \neq \gamma\in \SL(2,\ZZ)$ be an element of finite order, $d<\ord(\gamma)$ and $E= \CC/(\ZZ\tau+\ZZ)$ an elliptic curve. There is a natural number $\kappa$ such that for any element $x\in E$ that is not torsion of order less than or equal to $\kappa$ the sum 
		    \begin{equation*}
				\sum_{i=0}^d \gamma^i x = 0 \in E
			\end{equation*}
			if and only if $d = \ord(\gamma)-1$.
		\end{corollary}

		\begin{definition}
		    Let $X\to\PP^1$ be an elliptic Jacobian isotrivial K3 surface. Then the minimal $\kappa$ fulfilling the conditions of the previous corollary for all $\gamma\in \SL(2,\ZZ)$ that occur in the local monodromy of a singular fiber of $X\to\PP^1$ is denoted by $\kappa_X$.
		\end{definition}
		
		\begin{proposition}
			Let $X\to\PP^1$ be an isotrivial K3 surface. Let $C\subset X$ be a quasi-torsion curve such that $D(C)$ contains no component that is torsion of order up to $\kappa_X$. Then the geometric genus $g(C)$ satisfies
			\begin{equation*}
				g(C) \ge (C.X_t-1)c(X\to\PP^1) -2.
			\end{equation*}
		\end{proposition}
		\begin{proof}
			We follow the idea of \cite{hassett2003potential} by calculating the ramification occuring at the singular fibers and then applying the Hurwitz formula.
			
			Let $\Delta\subset \PP^1$ be a small disc around a singular fiber such that $C$ is smooth over punctured disc $\Delta^*$. Pick a local branch $B$ of $C$. Then there are two cases:\bigskip\\
			\underline{\textit{Case 1: $B$ is not a section:}} Fix a fiber $X_t$ and a point $p \in B\cap X_t$. Moreover let $\gamma\in\SL(2,\ZZ)$ be a generator of the local monodromy group. By construction the point $q = \gamma.p - p\in J^0(X_t)$ is not zero. Applying $\gamma$ again yields $\gamma^i.q = \gamma^i.p - \gamma^{i-1}.p$ and therefore
			\begin{equation*}
				\gamma^i.p - p = \sum_{j=0}^{i-1} \gamma^j.q.
			\end{equation*}
			By Corollary \ref{cor:OrderOfPointInElliptic} the smallest $i>0$ such that $\gamma^i.p = p$ is equal to $\ord(\gamma)-1$. Thus the ramification contribution $e_i$ of this branch is $\ord(\gamma)-1$.
			\bigskip\\
			\underline{\textit{Case 2: $B$ is a section:}} Suppose there is another branch that is also a local section. This in turn would yield a local section of $D(C)$ as well and thus we have $\gamma.p = p$ for some $p\in J^0(X_t)$. But this is a contradiction to the previous corollary and the assumption that $D(C)$ contains no torsion of order up to $\kappa_X$. Hence there is at most one branch that is a local section.
			\bigskip\\
			To conclude, for one fixed degenerate fiber with local monodromy generated by $\gamma\in \SL(2,\ZZ)$, we have that the ramification contribution $e_i$ is greater or equal to $(X_t.C-1)\cdot \tfrac{\ord(\gamma)-1}{\ord(\gamma)}$. Hence, by the Hurwitz formula we get
			\begin{align*}
		        	2g(C)-2 &\ge (X_t.C)\cdot (2g(\PP^1)-2) + \sum_i e_i \\
		        	        &\ge-2(X_t.C)+ (X_t.C-1)\cdot (\tfrac{1}{2}n_0 + \tfrac{5}{6}n_2 + \tfrac{3}{4}n_3 + \tfrac{2}{3}n_4) \\
		        	        &= (X_t.C-1)c(X\to\PP^1)-2.\qedhere
			\end{align*}

		\end{proof}
        Now we are finally able to prove the last remaining part of Theorem \ref{thm:NQTCurvesExist}.
        \begin{proof}[Proof of Theorem \ref{thm:NQTCurvesExist} in the remaining case.] 
            Let $d_0$ be the index of $X$ and let $p\gg0$ be a prime. 
            By \cite[Chapter 11.5]{huybrechts2016lectures} we can choose a $p$-twist $Y\to \PP^1$ of $X\to\PP^1$, i.e., $J^p(Y)\cong X$ as an isomorphism of elliptic surfaces. 
            Then the index of $Y$ is $d_0p$. By \cite[Lemma 3.5]{bogomolov1999density} we can choose a rational curve $R\subset Y$ with $R.Y_t = d_0p$. 
            Suppose that this curve is quasi-torsion and denote $k = (\kappa_X!)^n$ for some $n\in\NN$. Recall from Section \ref{sub:compactifiedJacobian} that there is a multiplication map, i.e., 
            \begin{equation*}
                g_k: J^1(Y) \dashrightarrow J^k(Y).
            \end{equation*}
			Then taking the image of $R$ under this map yields that $R' = g_{k}(R)$ is a rational curve in $Y' = J^{k}(Y)$. Moreover as $\gcd(p,k) = 1$ we know that $R'.Y'_t \ge p$.
			For $n\in\NN$ big enough we can assume that $D(R')$ does not contain non-trivial torsion of order up to $\kappa_X$. 
			Then the previous proposition shows that $R'$ (and hence $R$) can not be rational for $p\gg 0$, a contradiction. 
			Therefore $R$ is not quasi-torsion and $g_p(R)\subset J^p(Y) \cong X$ gives the desired curve.
        \end{proof}
    \section{Producing curves with many singularities}  
        \label{sec:ProducingSingularities}
        In this section we will prove Theorem~\ref{thm:MainTheorem1}. The idea is to examine what happens to rational curves under self-rational maps. 
		The latter are constructed as follows. We define the map $g_n\colon  X\dashrightarrow J^n(X)$ as the composition of the identification $X\cong J^1(X)$ and the multiplication map $J^1(X)\dashrightarrow J^n(X)$
		\begin{equation*}
			g_n\colon  X\to J^1(X)  \dashrightarrow J^n(X).
		\end{equation*}
        We will then show that given a non-quasi-torsion curve $C\subset X$ the rational maps $g_n$ produce new curves $C'=g_n(C)\subset X$ such that $C'$ has many singularities.
		\begin{proposition}
			\label{prop:ProductionNewCurves}
			Let $X\to \PP^1$ be an elliptic K3 surface and $C$ be a curve with $C.X_t > 1$ that is non-quasi-torsion and such that $g_n|_C$ is a birational map to its image for all $n\equiv 1\mod d_0$. Then for every $n\in \NN$ there are curves $C_i\subset J^n(X)$ with a rational map $C\dashrightarrow C_i$ such that $C_i^2 \to \infty$.
		\end{proposition}
		\begin{proof}
			Let some open $V\subset U\subset \PP^1$ be given. First we will show that there is some $m\equiv 1 \mod d_0$ such that $D(C)_V$ has a component with an isolated torsion point of order $m$.
			
			Suppose the contrary, i.e., $D(C)_V$ does not contain a component with an isolated torsion point $p_0$ of order $m \equiv 1 \mod d_0$ for some $m$. By shrinking $V$ we may assume that $D(C)_V$ is \'{e}tale over $V$, $V$ is simply connected and $J^0(X)_V\to V$ is given by a standard model. Applying Proposition \ref{prop:CapTorsionEmpty} to all branches of $D(C)_V$ we get that the branches of $D(C)_V$ - and hence all components of $D(C)$ - are quasi-torsion, which is a contradiction. 
			
			Let $k\in \NN$ be given and choose $k$ disjoint analytically open sets $V_1,\ldots,V_k\subset U$. Then by the above there are $m_1,\ldots, m_k\equiv 1\mod d_0$ such that $D(C)$ has an isolated torsion point of order $m_i$ over some $t_i\in V_i$. Denote $m = \prod m_i$. 
			Then by assumption the map $C\dashrightarrow g_m(C)$ is birational. Therefore $g_m(C)$ has a singularity over $t_i$ for all $i$ as $g_m|C$ maps two points of $C$ over $t_i$ to the same point in $g_m(C)$ by construction, giving a locally reducible singularity.
			
			For the last statement let $n\in\NN$ be given. We observe that $nm \equiv n\mod d_0$. Then $g_{nm}(C)$ also has at least $k$ singularities and the isomorphism $J^{nm}(X)\cong J^n(X)$ gives the result.
		\end{proof}
		\begin{proof}[Proof of Theorem \ref{thm:MainTheorem1}]
		    Let $R\subset X$ be a non-quasi-torsion rational curve as constructed in Theorem \ref{thm:NQTCurvesExist}. As $R$ is non-torsion, the set 
			\begin{equation*}
				\{J^k(X)_t.g_k(R)\eft k\equiv 1\mod d_0\}
			\end{equation*} 
			attains a minimum greater than $1$ for some $k_0$ as otherwise $R$ would be torsion. Now replace $R$ with $g_{k_0}(R)$ via the isomorphism $J^k(X)\cong J^1(X)$. Then the previous Proposition~\ref{prop:ProductionNewCurves} applies: If $R\dashrightarrow g_k(R)$ is not birational for some $k$, then $J^k(X)_t.g_k(R)< X_t.R$, a contradiction.
		\end{proof}

	\section{Density of lifted rational curves in \texorpdfstring{$\PP(\Omega_X)$}{P(OmX)}}
		\label{sec:JetSpace}
		Let $X\to \PP^1$ be an elliptic K3 surface. In this section we will examine the density in the jet space $\PP(\Omega_X)$ for lifts of curves $C$ that are constructed similarly to those in the last section. Recall that the lift $j\colon\tilde{C}\to \PP(\Omega_X) = P(\mathcal{T}_X)$ is analytically given by the pushforward of the tangent vectors. Moreover by construction we get
		\begin{equation*}
				c_1(\oO_{\PP(\Omega_X)}(1)).j_*(\tilde{C}) \le 2g(C)-2.
		\end{equation*}
		Now we will investigate the behaviour of lifts of (rational) curves in the jetspace of an elliptic K3 surface $X\to \PP^1$. Denote its index by $d_0$ and fix a line bundle $\mM\in\Pic X$ of degree $d_0$.  Furthermore let $C\subset X$ be a non-quasi-torsion curve coming from Section \ref{sec:ExistenceRNQT}. For $n\in I = \{n\in\NN\eft n\equiv 1\mod d_0\}$ denote by $G_n\colon J^1(X)\dashrightarrow J^n(X)\to J^1(X)$ the multiplication map $J^1(X)\dashrightarrow J^n(X)$ composed with the isomorphism $J^n(X)\to J^1(X)$ induced by the line bundle $\mM$, i.e., fiberwise a line bundle $L\in \mathrm{Jac}^1 (X_t)$ gets mapped to
		\begin{equation*}
		    L\mapsto L^{\otimes n} \mapsto L^{\otimes n} \otimes \mM|_{X_t}^{\otimes-(n-1)/d_0}.
		\end{equation*}
		\begin{lemma}
		    Let $X\to\PP^1$ be an elliptic K3 surface and $\Delta\subset \PP^1$ simply connected such that
		    \begin{equation*}
		        \CC\times \Delta / (\ZZ\tau(t)+\ZZ, t) \cong J^0(X)_\Delta \to\Delta
		    \end{equation*}
            is a standard model.
            Then we may choose an isomorphism $J^1(X)_\Delta \to J^0(X)_\Delta$ such that under this identification $G_n$ is given by
            \begin{equation*}
				(z,t)\mapsto (nz,t).
			\end{equation*}
		\end{lemma}
		\begin{proof}	
		    The line bundle $\mM$ induces a section $S\subset J^{d_0}(X)$ and we denote by $H$ the preimage of $S_\Delta$ under the smooth multiplication map $J^1(X)_\Delta\to J^{d_0}(X)_\Delta$. Then $H$ decomposes into a disjoint union of $d_0^2$ branches and picking one branch $h$ induces an isomorphism $J^1(X)_\Delta \to J^0(X)_\Delta$: Every point $h_t$ of $h$ over $t\in \Delta$ corresponds to a line bundle $L$ on $X_t$ of degree $1$ such that $L^{\otimes d_0} = \mM|_{X_t}$ and substracting this line bundle fiberwise yields the desired map. Viewing $G_n$ as a map $J^0(X)_\Delta\to J^0(X)_\Delta$ via this isomorphism a line bundle $L'$ on $J^0(X_t)$ gets mapped to
			\begin{equation*}
			    L' \mapsto (L'\otimes L)^{\otimes n} \otimes \mM^{\otimes (n-1)/d_0} \otimes L^{-1}= L'^{\otimes n},
			\end{equation*}
			and we are done.
		\end{proof}
        \begin{remark}
            \label{rem:TangentSpace}
            
            Let $X=(\CC\times\Delta)/(\ZZ\tau(t)+\ZZ,t)\to\Delta$ be a standard model and $p = (x\tau(t)+y,t)\in X$ a point. Then we can naturally choose an isomorphism of the tangent spaces 
            \begin{equation*}
    				T_pX \cong T_{(x\tau(t)+y,t)}\CC\times \Delta \cong	\CC\times \CC.		
    		\end{equation*}
    		For a given deck transformation $(z,t)\mapsto(z+a\tau(t)+b,t)$ the induced isomorphism on $T_pX \cong \CC\times\CC$ is given by
    		\begin{equation*}
    			(z,t) \mapsto (z+a\partial_t\tau(t), t).
    		\end{equation*}
	    \end{remark}

		The multiplication map $G_n$ is very similar to the maps $g_n$ from the last section. The difference becomes necessary as we really need to consider \emph{self}-rational maps of K3 surfaces in the following.
	
		We will show that the union of curves $G_n(C)$ lifted to the jet space $\PP(\Omega_X)$ are Zariski-dense. In particular if we take any rational non-quasi-torsion rational multisection from Section \ref{sec:ExistenceRNQT} the following proves Theorem \ref{thm:MainTheorem2}.
		\begin{theorem}
			Let $X\to\PP^1$ be an elliptic K3 surface of index $d_0$ and $C$ be a non-quasi-torsion curve. Then the curves $G_n(C)\; (n\equiv 1\mod d_0)$ lifted to $\PP(\Omega_X)$ form a dense subset in the Zariski topology.
		\end{theorem}
		\begin{proof}
			Denote the projection by $\operatorname{pr}\colon  \PP(\Omega_X)\to X$.
			It suffices to show that given any open subset $V\subset U\subset \PP^1$ there is a point $p\in C_V$ such that $\operatorname{pr}^{-1}(p)$ intersects the union of the lifts of the $G_n(C)$ at infinitely many points.
			
			By shrinking $V$ we may assume by the previous lemma that $X$ is given by a standard model 
			\begin{equation*}
			    X \cong (\CC\times\Delta)/(\ZZ\tau(t)+\ZZ,t),
			\end{equation*}
			the map $G_n$ is given by $(z,t)\mapsto (nz,t)$, and $C$ is smooth and locally given by $(f(t),t)$ for some holomorphic function $f\colon \Delta\to \CC$.
			As $C$ is non-quasi-torsion by assumption the curve $G_{d_0}(C)$ is non-quasi-torsion as well and we can apply that its torsion points are dense, see Corollary \ref{cor:TorsionDense}. This means that there exist $t_j\in V$ and $n_j\in I = \{n\in \NN\eft n\equiv 1\mod d_0\}$ such that
			\begin{equation}
				\label{eq:TosionOfF}
				(n_j-1)f(t_j) = a_j\tau(t_j)+b_j
			\end{equation}
			for some $a_j,b_j\in\ZZ$. Then for every $k\in \NN$ the $n_j^kf(t_j)$ satisfy
			\begin{equation*}
				n_j^k f(t_j) = n_j^{k-1}(a_j\tau(t_j)+b_j) + n_j^{k-1} f(t_j) =  f(t_j) + (a_j\tau(t_j)+b_j)\tfrac{1-n_j^{k}}{1-n_j}, %f(t_j) + (a_j+b_j\tau(t_j))\sum_{i=0}^k n_j^i =
			\end{equation*}
			where the last equality follows by induction.
			Therefore for $p,q>0$ the curve $G_{n_j^p}(C)$ intersects $G_{n_j^q}(C)$ over $t_j\in V$. Assume that for almost all indices $j$ there exist $p>q>0$ such that the tangent directions of $G_{n_j^p}(C)$ and $G_{n_j^q}(C)$ are the same over $t_j$. Then by Remark~\ref{rem:TangentSpace} this is equivalent to
			\begin{equation*}%TODO
				n_j^q\partial_t f(t_j)-a_j\tfrac{n_j^q-1}{n_j-1}\partial_t \tau(t_j) = n_j^p\partial_t f(t_j)- a_j\tfrac{n_j^p-1}{n_j-1}\partial_t \tau(t_j).
			\end{equation*}
			In other words
			\begin{equation*}
				\partial_t f(t_j) = \tfrac{a_j}{n_j-1}\partial_t\tau(t_j)
			\end{equation*}
			independently of $p,q$. 
			
			In the isotrivial case this means that $f$ is constant as $\partial_t \tau = 0$. If this was the case for all branches of $C$ over $V$ then the curve would be quasi-torsion, a contradiction. 
			
			In the non-isotrivial case the holomorphic function $\frac{\partial_t f}{\partial_t \tau}$ maps to $\RR$. Therefore it is constant as well and $a = \tfrac{a_j}{n_j-1}$ does not depend on $j$. Then on the other hand equation \eqref{eq:TosionOfF} yields that the holomorphic function $f-a\tau$ also maps to $\RR$ and therefore $b= \tfrac{b_j}{n_j-1}$ is independent of $j$ as well. But this in turn yields that $C$ is quasi-torsion if this was the case for every branch over $V$, and hence we are done.
		\end{proof}
	\section{Applications}
		\label{sec:Applications}
		\noindent As we will see the last section provides a simple tool to prove Kobayashi's Theorem in the special case of elliptic K3 surfaces.
		\begin{corollary}[Kobayashi's Theorem]
			\label{thm:Kobayashi}
			Let $X$ be an elliptic K3 surface. Then 
			\begin{equation*}
				\hh^0(X,\operatorname{Sym}^n \Omega_X) = 0
			\end{equation*} for all $n>0$.
		\end{corollary}
		\begin{proof}
			Let $\PP(\Omega_X)$ be the first jet-space of $X$. Then we have the equality
			\begin{equation*}
				\hh^0(\PP(\Omega_X), \oO(n)) = \hh^0(X,\operatorname{Sym}^n \Omega_X).
			\end{equation*} But we know from the last section that there are rational curves $R_i\subset X$ such that the union of their lifts is Zariski-dense in the jet space. But by construction $c_1(\oO(n)).R_i< 0$ and hence $\oO(n)$ is not effective.
		\end{proof}
		\noindent We would also like to mention the following corollary on the density of rational curves for all elliptic K3 surfaces in the usual topology. For a Baire-general K3 surface this was achieved in \cite{chen2013density}. Moreover in {\itshape loc.\ cit.}\ the following theorem has been proven:
		\begin{theorem}[{\cite[Theorem 1.6]{chen2013density}}]
			Let $X\to\PP^1$ be an elliptic K3 surface. If there is a non-torsion rational multisection then the union of rational curves is dense in the usual topology.
		\end{theorem}
		\noindent Using Theorem~\ref{thm:NQTCurvesExist} we directly get the following stronger result:
		\begin{corollary}
			\label{cor:DensityUsual}
			Let $X\to\PP^1$ be an arbitrary projective elliptic K3 surface. Then the union of rational curves is dense in the usual topology.
		\end{corollary}

	\bibliography{main2}{} 
	\bibliographystyle{plain}
	\vspace{3mm}
	\footnotesize\textsc{Georg-August-Universität Göttingen, Mathematisches Institut, Bunsenstr. 3-5, 37073 Göttingen, Germany}\\
	\textit{Email address: }\href{mailto:jonas.baltes@mathematik.uni-goettingen.de}{\nolinkurl{jonas.baltes@mathematik.uni-goettingen.de}}
	
\end{document}